\documentclass[12pt]{amsart}
 \usepackage[margin=1in]{geometry}
 \usepackage{amsmath,amsfonts,amsthm,amssymb,bbm}
 \usepackage{graphicx,color,dsfont}
 \usepackage{enumitem}
 \usepackage{fourier}

\begin{document}

\newtheorem{theorem}{Theorem}
\newtheorem{lemma}{Lemma}
\newtheorem{proposition}{Proposition}
\newtheorem{rmk}{Remark}
\newtheorem{example}{Example}
\newtheorem{exercise}{Exercise}
\newtheorem{definition}{Definition}
\newtheorem{corollary}{Corollary}
\newtheorem{notation}{Notation}
\newtheorem{claim}{Claim}
\newtheorem{Conjecture}{Conjecture}

\newtheorem{dif}{Definition}

 \newtheorem{thm}{Theorem}[section]
 \newtheorem{cor}[thm]{Corollary}
 \newtheorem{lem}[thm]{Lemma}
 \newtheorem{prop}[thm]{Proposition}
 \theoremstyle{definition}
 \newtheorem{defn}[thm]{Definition}
 \theoremstyle{remark}
 \newtheorem{rem}[thm]{Remark}
 \newtheorem*{ex}{Example}
 \numberwithin{equation}{section}

\newcommand{\vertiii}[1]{{\left\vert\kern-0.25ex\left\vert\kern-0.25ex\left\vert #1
    \right\vert\kern-0.25ex\right\vert\kern-0.25ex\right\vert}}

\newcommand{\R}{{\mathbb R}}
\newcommand{\C}{{\mathbb C}}
\newcommand{\U}{{\mathcal U}}
\newcommand{\norm}[1]{\left\|#1\right\|}
\renewcommand{\(}{\left(}
\renewcommand{\)}{\right)}
\renewcommand{\[}{\left[}
\renewcommand{\]}{\right]}
\newcommand{\f}[2]{\frac{#1}{#2}}
\newcommand{\im}{i}
\newcommand{\cl}{{\mathcal L}}
\newcommand{\ck}{{\mathcal K}}

\newcommand{\al}{\alpha}
\newcommand{\be}{\beta}
\newcommand{\wh}[1]{\widehat{#1}}
\newcommand{\ga}{\gamma}
\newcommand{\Ga}{\Gamma}
\newcommand{\de}{\delta}
\newcommand{\ben}{\beta_n}
\newcommand{\De}{\Delta}
\newcommand{\ve}{\varepsilon}
\newcommand{\ze}{\zeta}
\newcommand{\Th}{\Theta}
\newcommand{\ka}{\kappa}
\newcommand{\la}{\lambda}
\newcommand{\laj}{\lambda_j}
\newcommand{\lak}{\lambda_k}
\newcommand{\La}{\Lambda}
\newcommand{\si}{\sigma}
\newcommand{\Si}{\Sigma}
\newcommand{\vp}{\varphi}
\newcommand{\om}{\omega}
\newcommand{\Om}{\Omega}
\newcommand{\ra}{\rightarrow}

\newcommand{\ro}{{\mathbf R}}
\newcommand{\rn}{{\mathbf R}^n}
\newcommand{\rd}{{\mathbf R}^d}
\newcommand{\rmm}{{\mathbf R}^m}
\newcommand{\rone}{\mathbb R}
\newcommand{\rtwo}{\mathbb R^2}
\newcommand{\rthree}{\mathbf R^3}
\newcommand{\rfour}{\mathbf R^4}
\newcommand{\ronen}{{\mathbf R}^{n+1}}
\newcommand{\ku}{\mathbf u}
\newcommand{\kw}{\mathbf w}
\newcommand{\kf}{\mathbf f}
\newcommand{\kz}{\mathbf z}
\newcommand{\bb}{\mathbf B}

\newcommand{\N}{\mathbf N}

\newcommand{\tn}{\mathbf T^n}
\newcommand{\tone}{\mathbf T^1}
\newcommand{\ttwo}{\mathbf T^2}
\newcommand{\tthree}{\mathbf T^3}
\newcommand{\tfour}{\mathbf T^4}

\newcommand{\zn}{\mathbf Z^n}
\newcommand{\zp}{\mathbf Z^+}
\newcommand{\zone}{\mathbf Z^1}
\newcommand{\zz}{\mathbf Z}
\newcommand{\ztwo}{\mathbf Z^2}
\newcommand{\zthree}{\mathbf Z^3}
\newcommand{\zfour}{\mathbf Z^4}

\newcommand{\hn}{\mathbf H^n}
\newcommand{\hone}{\mathbf H^1}
\newcommand{\htwo}{\mathbf H^2}
\newcommand{\hthree}{\mathbf H^3}
\newcommand{\hfour}{\mathbf H^4}

\newcommand{\cone}{\mathbf C^1}
\newcommand{\ctwo}{\mathbf C^2}
\newcommand{\cthree}{\mathbf C^3}
\newcommand{\cfour}{\mathbf C^4}
\newcommand{\dpr}[2]{\langle #1,#2 \rangle}

\newcommand{\sn}{\mathbf S^{n-1}}
\newcommand{\sone}{\mathbf S^1}
\newcommand{\stwo}{\mathbf S^2}
\newcommand{\sthree}{\mathbf S^3}
\newcommand{\sfour}{\mathbf S^4}

\newcommand{\lp}{L^{p}}
\newcommand{\lppr}{L^{p'}}
\newcommand{\lqq}{L^{q}}
\newcommand{\lr}{L^{r}}
\newcommand{\echi}{(1-\chi(x/M))}
\newcommand{\chip}{\chi'(x/M)}

\newcommand{\wlp}{L^{p,\infty}}
\newcommand{\wlq}{L^{q,\infty}}
\newcommand{\wlr}{L^{r,\infty}}
\newcommand{\wlo}{L^{1,\infty}}

\newcommand{\lprn}{L^{p}(\rn)}
\newcommand{\lptn}{L^{p}(\tn)}
\newcommand{\lpzn}{L^{p}(\zn)}
\newcommand{\lpcn}{L^{p}(\cn)}
\newcommand{\lphn}{L^{p}(\cn)}

\newcommand{\lprone}{L^{p}(\rone)}
\newcommand{\lptone}{L^{p}(\tone)}
\newcommand{\lpzone}{L^{p}(\zone)}
\newcommand{\lpcone}{L^{p}(\cone)}
\newcommand{\lphone}{L^{p}(\hone)}

\newcommand{\lqrn}{L^{q}(\rn)}
\newcommand{\lqtn}{L^{q}(\tn)}
\newcommand{\lqzn}{L^{q}(\zn)}
\newcommand{\lqcn}{L^{q}(\cn)}
\newcommand{\lqhn}{L^{q}(\hn)}

\newcommand{\lo}{L^{1}}
\newcommand{\lt}{L^{2}}
\newcommand{\li}{L^{\infty}}
\newcommand{\beqn}{\begin{eqnarray*}}
\newcommand{\eeqn}{\end{eqnarray*}}
\newcommand{\pplus}{P_{Ker[\cl_+]^\perp}}

\newcommand{\co}{C^{1}}
\newcommand{\coi}{C_0^{\infty}}

\newcommand{\ca}{\mathcal A}
\newcommand{\cs}{\mathcal S}
\newcommand{\cm}{\mathcal M}
\newcommand{\cf}{\mathcal F}
\newcommand{\cb}{\mathcal B}
\newcommand{\ce}{\mathcal E}
\newcommand{\cd}{\mathcal D}
\newcommand{\cg}{\mathcal G}
\newcommand{\cn}{\mathcal N}
\newcommand{\cz}{\mathcal Z}
\newcommand{\crr}{\mathbf R}
\newcommand{\cc}{\mathcal C}
\newcommand{\ch}{\mathcal H}
\newcommand{\cq}{\mathcal Q}
\newcommand{\cp}{\mathcal P}
\newcommand{\cx}{\mathcal X}
\newcommand{\eps}{\epsilon}

\newcommand{\pv}{\textup{p.v.}\,}
\newcommand{\loc}{\textup{loc}}
\newcommand{\intl}{\int\limits}
\newcommand{\iintl}{\iint\limits}
\newcommand{\dint}{\displaystyle\int}
\newcommand{\diint}{\displaystyle\iint}
\newcommand{\dintl}{\displaystyle\intl}
\newcommand{\diintl}{\displaystyle\iintl}
\newcommand{\liml}{\lim\limits}
\newcommand{\suml}{\sum\limits}
\newcommand{\ltwo}{L^{2}}
\newcommand{\supl}{\sup\limits}
\newcommand{\df}{\displaystyle\frac}
\newcommand{\p}{\partial}
\newcommand{\Ar}{\textup{Arg}}
\newcommand{\abssigk}{\widehat{|\si_k|}}
\newcommand{\ed}{(1-\p_x^2)^{-1}}
\newcommand{\tT}{\tilde{T}}
\newcommand{\tV}{\tilde{V}}
\newcommand{\wt}{\widetilde}
\newcommand{\Qvi}{Q_{\nu,i}}
\newcommand{\sjv}{a_{j,\nu}}
\newcommand{\sj}{a_j}
\newcommand{\pvs}{P_\nu^s}
\newcommand{\pva}{P_1^s}
\newcommand{\cjk}{c_{j,k}^{m,s}}
\newcommand{\Bjsnu}{B_{j-s,\nu}}
\newcommand{\Bjs}{B_{j-s}}
\newcommand{\Ly}{L_i^y}
\newcommand{\dd}[1]{\f{\partial}{\partial #1}}
\newcommand{\czz}{Calder\'on-Zygmund}
\newcommand{\chh}{\mathcal H}

\newcommand{\lbl}{\label}
\newcommand{\beq}{\begin{equation}}
\newcommand{\eeq}{\end{equation}}
\newcommand{\beqna}{\begin{eqnarray*}}
\newcommand{\eeqna}{\end{eqnarray*}}
\newcommand{\bp}{\begin{proof}}
\newcommand{\ep}{\end{proof}}
\newcommand{\bprop}{\begin{proposition}}
\newcommand{\eprop}{\end{proposition}}
\newcommand{\bt}{\begin{theorem}}
\newcommand{\et}{\end{theorem}}
\newcommand{\bex}{\begin{Example}}
\newcommand{\eex}{\end{Example}}
\newcommand{\bc}{\begin{corollary}}
\newcommand{\ec}{\end{corollary}}
\newcommand{\bcl}{\begin{claim}}
\newcommand{\ecl}{\end{claim}}
\newcommand{\bl}{\begin{lemma}}
\newcommand{\el}{\end{lemma}}
\newcommand{\dea}{(-\De)^\be}
\newcommand{\naa}{|\nabla|^\be}
\newcommand{\cj}{{\mathcal J}}
\newcommand{\ci}{{\mathcal I}}
\newcommand{\ubb}{{\mathbf u}}

\title[A note on the instability of the Ruf-Sani solitons]{On the instability of the Ruf-Sani solitons for the NLS with exponential nonlinearity}

\author[Hichem Hajaiej]{\sc Hichem Hajaiej}
\address{
Department of Mathematics, College of  Natural Science
Cal State University, 5151 State Drive, 90032 Los Angeles, California, USA,
}
\email{hhajaie@calstatela.edu}

\author[Atanas G. Stefanov]{\sc Atanas G. Stefanov}
\address{Department of Mathematics, University of Alabama - Birmingham, 
	1402 10th Avenue South
	Birmingham AL 35294, USA
 }
\email{stefanov@uab.edu}

 \thanks{
 Atanas Stefanov acknowledges  partial support  from NSF-DMS,   \# 1908626.}

\subjclass[2010]{Primary: 35 Q 55, 35 B 44, 37 K 40, 37 K 45}

\keywords{NLS with exponential nonlinearity, solitons, instability}

\date{\today}

\begin{abstract}
We study the two dimensional non-linear Schr\"odinger equation with two types of exponential non-linearities. It is well-known by a work of Ruf - Sani, \cite{RS}, that such models support solitary wave solutions, which are solutions of some constrained minimization problem. We show that these Ruf - Sani solitons, are spectrally unstable as well as unstable by blow-up.
\end{abstract}

\maketitle
\section{ Introduction}
 We consider the Schr\"odinger equation with focusing exponential nonlinearity
 \begin{equation}
 \label{10}
 i u_t +\De u+ f_\mu(u) =0; \ \ (t,x)\in \rone\times \rtwo
 \end{equation}
 with
 $$
 f_\mu(u)=(e^{4\pi |u|^2} - 1 - 4\pi \mu |u|^2) u, \ \  \mu\in \{0,1\}.
 $$
  Note that this model enjoys the conserved quantities
  \begin{eqnarray*}
  M(u) &=& \int_{\rtwo} |u|^2 dx \\
  E_\mu(u) &=& \f{1}{2} \int_{\rtwo} |\nabla u|^2 dx- \int_{\rtwo} F_\mu(u) dx,
  \end{eqnarray*}
 where $F_\mu'=f_\mu, F_\mu(0)=0$. Explicitly,
 $$
 F_\mu(u) =\f{1}{8\pi} \left(e^{4\pi |u|^2} - 1 - 4\pi |u|^2-8\pi^2 \mu |u|^4\right), \ \  \mu\in \{0,1\}.
 $$
 Our work concentrates on the solitary waves of \eqref{10}, namely the solutions in the form $u=e^{i \om t} \phi, \om>0$, which clearly satisfy the profile problem
 \begin{equation}
 \label{20}
 -\De \phi+\om \phi = f_\mu(\phi).
 \end{equation}
 \subsection{Ground states}
 There are various definitions of ground states, which may be adopted for such objects. The notion of ground state has to do with an underlying mode of  variational construction. In our case, we shall exclusively consider the Ruf-Sani construction, \cite{RS}, which solves a particular  constrained variational problem. Here is the precise result, due to Ruf-Sani, \cite{RS}.
 \begin{proposition}(\cite{RS}, see also \cite{DKM})
 	\label{prop:10}
 	Let $f:\rone_+\to \rone_+$ be a continuous function, which  satisfies
 	
 		$$
 		\lim_{|t|\to \infty} \f{f(t)}{e^{\al t^2}}= \left\{
 		\begin{array}{cc}
 	0 & \al>4\pi \\
 	+\infty & \al<4\pi
 		\end{array}
 		\right.
 		$$
 		
 		$$
 		\lim_{t\to 0} \f{f(t)}{t}=0, \ \  \limsup_{|t|\to +\infty} \frac{t f(t)}{e^{4\pi t^2}}>0.
 		$$

 		 For $F: F'=f, F(0)=0$, and all $s\neq 0$,
 		$$
 		0<2 F(s)\leq s f(s).
 		$$
 	
 	Then, the minimization problem
 	\begin{equation}
 	\label{16}
 	\left\{
 	\begin{array}{l}
 	\|\nabla u\|\to \min \\
 	\textup{subject to}\ \ \frac{1}{2} \|u\|^2 - \int_{\rtwo}  F(u(x)) dx = 0.
 	\end{array}
 	\right.
 	\end{equation}
 	has a solution $Q$. Moreover, $Q$ satisfies the following properties:
 	\begin{itemize}
 		\item $Q$ solves the Euler-Lagrange equation
 		\begin{equation}
 		\label{30}
 		- \De Q+ Q = f(Q)
 		\end{equation}
 		\item $Q$ is radially symmetric, $Q\in C^2\cap L^\infty$, $Q$ is exponentially decaying at $\pm \infty$.
 		\item $0<\|\nabla Q\|<1$ and
 		\begin{eqnarray}
 		\label{38}
 	& & \frac{1}{2} \|Q\|^2 = 	\int_{\rtwo}  F(Q), \\
 		\label{40}
 	& & 	\|\nabla Q\|^2 + \|Q\|^2 = \int_{\rtwo}  f(Q) Q.
 		\end{eqnarray}
 	\end{itemize}
 \end{proposition}
 {\bf Remark:} We would like to note that if one starts with a nice solution of the elliptic problem \eqref{30}, then the relation \eqref{38} is nothing but the Pohozaev identity for such solutions and can be easily obtained by integration by parts, by taking dot product of \eqref{30} with $x\cdot \nabla Q$. Similarly, \eqref{40} is obtained after taking dot product of \eqref{30} with $Q$. In

As a simple consequence of this result, we will obtain suitable solutions of \eqref{20}. Indeed, for a fixed $\om>0$, set
$$
f(z):=\f{1}{\om} (e^{4\pi z^2}-1-4\pi \mu z^2)z, \mu\in \{0,1\},
$$
 with the corresponding function
 $$
 F(z)=\f{1}{8\pi \om} \left(e^{4\pi |u|^2} - 1 - 4\pi |u|^2-8\pi^2 \mu |u|^4\right).
 $$
 We claim that the pair $f, F, \mu\in \{0,1\}$ satisfies the conditions in Proposition \ref{prop:10}. For the case, $\mu=0$, the only non trivial part of this statement is the inequality $2 F(z)\leq z f(z)$, which can be seen by the expansion in McLaurin series
 $$
 2 F(z)=\f{1}{4\pi \om} \sum_{l=2}^\infty \f{(4\pi z^2)^l}{ l!} <
\f{1}{4\pi \om}  \sum_{l=2}^\infty \f{(4\pi z^2)^l}{(l-1)!} = z f(z)
 $$
 and similar for the case $\mu=1$.
We can thus infer the existence of a function $Q_\om$, as specified in Proposition \ref{prop:10}. Moreover, the assignment $\phi_\om(x):=Q_\om(\sqrt{\om} x)$ introduces a function, which is a solution of \eqref{20}, since $Q_\om$ solves the Euler-Lagrange equation \eqref{30}, corresponding to the specific nonlinearity $f_\om$.  We need a few definitions.

For a Lebesgue measurable function $f:\mathbb{R}^{n}\to \mathbb{R}_{+}$, we say that it is vanishing at infinity if
$\text{mes}\left\{ x\in \mathbb{R}^{n}:\ f(x)>a\right\} <\infty$ for any $a>0$; $d_{f}(a)=\text{mes}\left\{ x\in \mathbb{R}^{n}:\ f(x)>a\right\}$.\\
If $f$ vanishes at infinity, we denote $f^{*}$ the rearrangement of $f$. It is the unique function satisfying the following properties:
$$\begin{cases}
d_{f}(a)=d_{f^{*}}(a)\quad \text{for any}\quad a>0.\\
\text{There exists}\quad \rho:(0,\infty)\to \mathbb{R}_{+}\quad \text{a right continuous}\\
\text{ function such that}\quad f^{*}(x)=\rho(|x|).
\end{cases}$$

We say that $f$ is Schwarz symmetric (bell-shaped) if $f=f^{*}.$
Using rearrangement inequalities proven in \cite[Theorem 4.4]{Hajaiej-Stuart}.
We know that for any $u\in H^{1}_{+}(\mathbb{R}^{2})=\left\{ H^{1}(\mathbb{R}^{2}):\ u\geq 0\right\} $:
\begin{equation}\label{R1}
\int_{\mathbb{R}^{2}}(e^{4\pi|u|^{2}}-1)dx=\int_{\mathbb{R}^{2}}(e^{4\pi|u^{*}|^{2}}-1)dx
\end{equation}
\begin{equation}\label{R2}
\int_{\mathbb{R}^{2}}| u|^{4} dx=\int_{\mathbb{R}^{2}}| u^{*}|^{4} dx.
\end{equation}
On the other hand, thanks to the Polya-Szeg\"o inequality
\begin{equation}\label{R3}
\|\nabla u^{*}\|^{2}_{2}\leq \|\nabla u\|^{2}_{2}.
\end{equation}
If $f$ is not a non-negative function \eqref{R1} and \eqref{R2} remain true,\eqref{R3} becomes
\begin{equation}\label{R31}
\|\nabla f^{*}|\|_{2}\leq   \|\nabla  f\|_{2}.
\end{equation}
$(\ref{R1})$, $(\ref{R2})$ and $(\ref{R31})$ imply that the minimization problem (Ruf-Sani)
$$\inf\left\{ \|\nabla  u\|_{2}^{2}:\ \frac{1}{2}\| u\|_{2}^{2}-\int_{\mathbb{R}^{2}}F_{\mu}(u(x))dx=0\right\} $$
always have a Schwarz-symmetric (bell-shaped) minimizer.
We are now ready to collect our findings about $\phi_\om$  in the following corollary.
\begin{corollary}
	\label{cor:20}
	For each $\om>0$ and $\mu\in \{0,1\}$, there exists a solution $\phi_\om$ of the elliptic problem \eqref{20}. Moreover,
	\begin{itemize}
		\item $\phi_\om \in C^2\cap L^\infty$ is a bell-shaped function,
		\item $0<\|\nabla \phi_\om\|_{L^2}<1$ and
		\begin{eqnarray}
		 	\label{42}
		 	& & \frac{\om}{2} \|\phi_\om\|^2 = 	\int_{\rtwo}  F(\phi_\om), \\
		 	\label{45}
		 	& & 	\|\nabla \phi_\om\|^2 + \om \|\phi_\om\|^2 = \int_{\rtwo}  f(\phi_\om) \phi_\om.
		\end{eqnarray}
	\end{itemize}
\end{corollary}
{\bf Remark:} We call the functions $\phi_\om$ the {\bf Ruf-Sani solitons} associated with the nonlinear \\ Schr\"odinger equation with exponential nonlinearity \eqref{10}.
 \subsection{Main results}
 The main objective of this paper is to study further properties of the Ruf-Sani solitons. It is for example easy to compute the precise asymptotics at $\pm \infty$. Namely, it is a standard to obtain
 \begin{equation}
 \label{58}
  \phi_\om(x)= c \f{e^{-\sqrt{\om} |x|}}{\sqrt{|x|}}+o\left(\f{e^{-\sqrt{\om} |x|}}{|x|}\right), |x|>>1, x\in \rtwo
 \end{equation}
see for example Theorem 2, \cite{St1}, which works for general super-linear nonlinearity.

Next, we shall be interested in the properties of the linearized operators $\cl_\pm$. For convenience, introduce functions $g, G$
\begin{eqnarray}
\label{g10}
	g(z) &=&e^{4\pi z}-1-4\pi \mu z \\
	\label{g11}
	G(z) &=& \f{e^{4\pi z}-1- 4\pi z - 8 \pi^2 \mu z^2}{4\pi},
\end{eqnarray}
so that $f(z)=g(z^2) z$, $G(z^2)=2F(z)$ and  $G(0)=0, G'(z)= g(z)$.
In these  variables,
\begin{eqnarray*}
\cl_- &=& - \De+\om - g(\phi^2_\om)=- \De+\om - (e^{4\pi \phi_\om^2}-1-4\pi \mu \phi_\om^2); \\
\cl_+ &=& - \De+\om - (2 \phi^2 g'(\phi^2)+ g(\phi^2) )= - \De+\om -   ( e^{4\pi \phi_\om^2}(8 \pi \phi_\om^2+1) - 1 - 12 \mu \pi \phi_\om^2).
\end{eqnarray*}
 as these are paramount in the stability analysis of the waves $\phi_\om$ as solutions to \eqref{10}, see \eqref{50} below.

 In line with the expectations in the classical cases of power nonlinearities, we have the usual properties of $\cl_\pm$. Recall that for a semi-bounded from below self-adjoint operator $S$ with a finite dimensional negative subspace $X_-$,
 the Morse index is defined as follows $n(S)=dim(X_-)=\#\{\si_{p}(S)\cap (-\infty, 0)\}$.
 \begin{proposition}
 	\label{theo:10}
 Let $\om>0, \mu\in \{0,1\}$ and $\phi_\om$ are the Ruf-Sani solitons constructed in Corollary \ref{cor:20}. Then, 	the Schr\"odinger operators $\cl_\pm$ have the following properties
 \begin{itemize}
 	\item $\cl_-\geq 0$, with a simple eigenvalue at zero, $Ker[\cl_-]=span[\phi_\om]$.
 	\item $\cl_+$ has Morse index $1$. That is, $n(\cl_+)=1$.
 \end{itemize}
 \end{proposition}
 Our next result concerns the instability of the Ruf-Sani waves. In order to put the results in the proper context, let us consider the linearization of the Schr\"odinger equation with exponential nonlinearity in a vicinity of the soliton $e^{i \om t} \phi_\om$. More precisely, take
 $u=e^{ i \om t} (\phi_\om(x)+ v)= e^{ i \om t} (\phi_\om(x)+ v_1+i v_2)$ and plug this in \eqref{10}. After ignoring all terms $O(|v|^2)$, we obtain the linearized system
 \begin{equation}
 \label{50}
 \p_t \left(\begin{array}{c}
 v_1 \\ v_2
 \end{array}\right)= \left(\begin{array}{cc}
 0 & 1 \\ -1 & 0
 \end{array}\right) \left(\begin{array}{cc}
 \cl_+  & 0 \\ 0 & \cl_-
 \end{array}\right) \left(\begin{array}{c}
 v_1 \\ v_2
 \end{array}\right)=:\cj \cl \vec{v}.
 \end{equation}
 Passing to a time independent problem, $\vec{v} \to e^{\la t} \vec{v}$, allows us to reduce matters to the eigenvalue problem
 \begin{equation}
 \label{55}
 \cj \cl \vec{v} =\la \vec{v}.
 \end{equation}
 We now give the standard definitions of instability, as all our results refer to unstable waves.
 \begin{definition}
 	\label{defi:20}
 	
 We say that the wave $\phi_\om$ is spectrally unstable, if the eigenvalue problem \eqref{55} has a non-trivial solution $(\la, \vec{v}): \Re \la>0, \vec{v} \in D(\cl), \vec{v}\neq 0$.\\

 We say that the wave $\phi_\om$ is orbitally unstable, if there exists $\eps_0>0$ and a sequence $\vp_n: \lim_n \|\vp_n-\phi_\om\|_{H^1}=0$, while for the corresponding solutions of \eqref{10}, with initial Cauchy data $\vp_n$, we have
 $$
 \sup_{t>0} \inf_{y\in \rtwo, \theta\in\rone} \|\vp_n(t,\cdot) -
 e^{i \theta} \phi_\om(\cdot-y)\|_{L^2}\geq \eps_0.
 $$
 We say that $\phi_\om$ is unstable by blow up, if there exists  a sequence $\vp_n: \lim_n \|\vp_n-\phi_\om\|_{H^1}=0$, so that the corresponding solutions of \eqref{10}, with initial Cauchy data $\vp_n$, blow up in finite time.
 \end{definition}
 The next theorem is the main result of our work.
 \begin{theorem}
 	\label{theo:20}
 Let $\om>0, \mu\in \{0,1\}$. 	Then, Ruf-Sani  waves $e^{i \om t} \phi_\om$ are spectrally unstable, with a single real growing mode.  Moreover, these waves are unstable by blow-up.
 \end{theorem}

 \section{Proof of Proposition  \ref{theo:10}}
 We start by establishing  an alternative  characterization of   $\phi_\om$,  which will enable us to derive the spectral properties of $\cl_+$.
  \subsection{An alternative variational characterization of $\phi_\om$}
 Specifically, taking into account  that $Q_\om$ is a constrained minimizer for \eqref{16}, it is easy to see,  by rescaling,  that $\phi_\om$ is a constrained minimizer of
 	\begin{equation}
 	\label{17}
 	\left\{
 	\begin{array}{l}
 	\|\nabla u\|_{L^2} \to \min \\
 	\textup{subject to}\ \ \om  \|u\|^2 - \f{1}{4\pi} \int_{\rtwo} \left(e^{4\pi |u|^2}- 1 - 4\pi |u|^2 - 8\pi^2 \mu |u|^4\right) dx = 0.
 	\end{array}
 	\right.
 	\end{equation}
 We will show that reversing the roles of the constraints and the cost function produces the same outcome.
 \begin{lemma}
 	\label{le:10}
 	The constrained minimization problem
 		\begin{equation}
 		\label{18}
 		\left\{
 		\begin{array}{l}
 		I[u]:=\om \|u\|^2 - \f{1}{4\pi} \int_{\rtwo} \left(e^{4\pi |u|^2}- 1 - 4\pi |u|^2 - 8\pi^2 \mu |u|^4\right) dx \to \min \\
 		\textup{subject to}\ \  \|\nabla u\|_{L^2}=\|\nabla \phi_\om\|_{L^2}
 		\end{array}
 		\right.
 		\end{equation}
 		has a solution and
 		$
 		I_{\min}:= \inf_{\|\nabla u\|=\|\nabla \phi_\om\|} I[u]= 0.
 		$
 		In particular,  $u=\phi_\om$ solves \eqref{18}.
 \end{lemma}
 \begin{proof}
 	The argument is pretty straightforward and exploits the fact that $\phi_\om$ is a solution of \eqref{17}. Indeed, as $u=\phi_\om$  satisfies the constraint of \eqref{18} and \eqref{17}, we have that
 	$$
 	I_{\min}:= \inf_{\|\nabla u\|=\|\nabla \phi_\om\|} I[u]\leq I[\phi_\om]=0
 	$$
 	Note that,  so far,  we have not even ruled out the scenario $I_{\min}=-\infty$!  We however claim that $I_{\min}=0$, which means that $u=\phi_\om$ is a solution to \eqref{18}.
 	
 	Assume, for a contradiction, that this is not the case. Then, there exists $\tilde{\phi}\neq 0$, so that $\|\nabla \tilde{\phi}\|=\|\nabla \phi_\om\|$, but $I[\tilde{\phi}]<0$. Consider then the continuous function $h(\al):=I[\al \tilde{\phi}]:[0,1]\to \rone$. Since\footnote{In the case $\mu=0$, the sum runs from $l=2$, while in the case $\mu=1$, from $l=3$}
 	$$
 h(\al)=\al^2\left[\om \|\tilde{\phi}\|^2-\f{1}{4\pi}  \sum_{l=2\ \ \textup{or} \ \ l=3}^\infty \al^{2 l-2} \int_{\rtwo} \f{(4\pi \tilde{\phi}^2)^l}{l!} \right],
 	$$
 	it is clear that $h(\al)>0$, for all $0<\al<<1$. Since $h(1)=I[\tilde{\phi}]<0$, it follows by continuity that for some $\tilde{\al}\in (0,1)$, we have that $I(\tilde{\al}\tilde{\phi})= h(\tilde{\al})=0$. Thus, $\tilde{u}:=\tilde{\al}\tilde{\phi}$ satisfies the constraints in \eqref{17}. But then, we reach a contradiction, as
 	$$
 	\|\nabla \phi_\om\|\leq \|\nabla \tilde{u}\|=\tilde{\al} \|\nabla \tilde{\phi}\|=\tilde{\al}
 		\|\nabla \phi_\om\|.
 	$$
 \end{proof}
 Our next task is to establish the spectral properties of $\cl_\pm$, based on the fact that $\phi_\om$ is a constrained minimizer of \eqref{18}.
 \subsection{The Morse index of  $\cl_+$ is exactly one}
We consider  a variation of the function $\phi_\om$ in \eqref{18}, which has a built in property $\|\nabla u\|=\|\nabla \phi_\om\|$. More specifically, for a test function $h$, consider
$$
u_\eps=\|\nabla \phi_\om\| \f{\phi_\om +\eps h}{\|\nabla(\phi_\om +\eps h)\|},
$$
 which satisfies the constraints of \eqref{18}. The function $m(\eps):=I[u_\eps]$ then has a minimum at $\eps=0$, with $m(0)=0$.  The necessary condition $m'(0)=0$ yields the Euler-Lagrange equation for this problem, which is, as expected, nothing but \eqref{20}. Note that this follows, as we take into account the relation  \eqref{45}. Next, the necessary condition $m''(0)\geq 0$ can be written explicitly as well. However, this gets a bit technical, so we reduce our considerations to the case $h\perp \De \phi_\om$, which is enough for our purposes, while simplifying the expressions. Indeed, we have
 $$
 \|\nabla(\phi_\om +\eps h)\|^2=\|\nabla \phi_\om\|^2- 2\eps \dpr{\De \phi_\om}{h}+\eps^2 \|\nabla h\|^2= \|\nabla \phi_\om\|^2 +\eps^2 \|\nabla h\|^2
 $$
 Also, note that since already we have ensured the validity of $m(0)=m'(0)=0$, it follows that $m(\eps)=const. \eps^2+o(\eps^2)$. Thus, we can ignore all powers of $\eps$ in the expansion of $m(\eps)$. To this end,
 \begin{eqnarray*}
& &  m(\eps) =  I[u_\eps]=\om \|u_\eps\|_{L^2}^2 - \int_{\rtwo} G(u_\eps^2)= \\
 &=& \om \f{\|\phi_\om\|^2+2 \eps \dpr{\phi_\om}{h}+\eps^2 \|h\|^2}{1+\f{\eps^2}{\|\nabla\phi_\om\|^2} \|\nabla h\|^2} - \int_{\rtwo} G\left(\phi_\om^2+2 \eps \phi_\om h+ \eps^2\left(h^2 - \f{\phi_\om^2}{\|\nabla \phi_\om\|^2} \|\nabla h\|^2\right)\right)=\\
 &=& \eps^2\left(\om \|h\|^2 - \om \f{\|\phi_\om\|^2}{\|\nabla \phi_\om\|^2} \|\nabla h\|^2 - \int_{\rtwo}
 G'(\phi_\om^2) \left(h^2 - \f{\phi_\om^2}{\|\nabla \phi_\om\|^2} \|\nabla h\|^2\right) - 2 \int_{\rtwo}G''(\phi_\om^2) \phi_\om^2 h^2\right)+O(\eps^3)\\
 &=& \eps^2 \left(\|\nabla h\|^2+\om \|h\|^2 - \dpr{2\phi^2 g'(\phi^2) + g(\phi^2)}{h} \right)+O(\eps^3)   \\
 &=& \eps^2 \dpr{\cl_+ h}{h}+O(\eps^3),
 \end{eqnarray*}
 where we took into account \eqref{45}. Clearly then, $\dpr{\cl_+ h}{h}=\f{m''(0)}{2}\geq 0$, for all $h\perp \De \phi_\om$. In particular, $n(\cl_+)\leq 1$. On the other hand,
 $$
 \dpr{\cl_+ \phi_\om}{\phi_\om}=\|\nabla \phi_\om\|^2+\om \|\phi_\om\|^2 - \int f(\phi_\om^2) - 8\pi \int (e^{4\pi \phi^2}-\mu) \phi_\om^4 dx = - 8\pi \int (e^{4\pi \phi_\om^2}-\mu) \phi_\om^4 dx<0,
 $$
 which shows by the min-max characterization of the eigenvalues,  that there is a negative eigenvalue, whence $n(\cl_+)=1$.
 \subsection{ $\cl_-\geq 0$ with a simple eigenvalue at zero}

  As $\cl_-$ is a Schr\"odinger operator with a radial kernel, its action,  can be decomposed, in a standard way, on the spaces of spherical harmonics. Indeed, $\De$ acts invariantly on each separate spherical harmonic space ${\mathcal X}_l, l=0, \ldots$ (composed of the eigenfunctions for a fixed eigenvalue $\la_l=-l^2$ for $\De_{\sone}$). So, denoting
 $$
 \cl_{-,l}:=\cl_+|_{{\mathcal X_l}}=-\p_{rr}-\f{1}{r} \p_r + \f{l^2}{r^2} -  g(\phi^2),
 $$
  we obtain $\cl_-=\oplus_{l=0}^\infty  \cl_{-,l}$, which can be thought of as acting on $L^2(r dr)$. Note that
 $$
 \cl_{-,0}<\cl_{-,1}<\ldots.
 $$
 Note that $\cl_-[\phi_\om]=0$, which implies that $\cl_{-,0}[\phi_\om]=0$. Since $\phi_\om>0$, this must be the ground state and the lowest eigenvalue for $\cl_{-,0}$ is zero. It follows that $\cl_-\geq \cl_{-,0}\geq 0$, with a simple eigenvalue at zero, spanned by $\phi_\om$.
 \section {Spectral  instability of the Ruf-Sani solitons}
  We start with an introduction to the instability index count, as developed in \cite{KKS, KKS2, KP}.
  \subsection{The Krein instability index theory}
  Consider the eigenvalue problem in the form
  \begin{equation}
  \label{551}
  J L f=\la f.
  \end{equation}
  Introduce the generalized eigenspace
  $$
  E_0=gKer[J L]:=span[\cup_{k=1}^\infty Ker (J L)^k],
  $$
  which turns out to be finite dimensional, under generic assumptions\footnote{certainly verified in our case}. Let $E_0=Ker[L]\oplus \tilde{E}_0$, and select a basis for it, say
  $\tilde{E}_0=span[\psi_1, \ldots, \psi_N]$. Next, consider the  symmetric matrix  $\cd=(\dpr{L \psi_i}{\psi_j})_{i,j\in (1,N)}$. Under these general assumptions,  it is proved in  \cite{KKS} (see Theorem 1), that
  \begin{equation}
  \label{se:10}
  k_r+2 k_c+ 2k_0^{\leq 0}= n(\cl)-n(\cd),
  \end{equation}
  where $k_r$ is the number of real and positive solutions $\la$ in \eqref{551}, which account for the real unstable modes, $2 k_c$ is the number of solutions $\la:\Re \la>0, \Im \la\neq 0$ in \eqref{551}, which account for the modulational instabilities, and finally $2k_0^{\leq 0}$ is the number of the dimension of the marginally  stable directions, corresponding to purely imaginary eigenvalue with negative Krein index.

  In our situation, namely the eigenvalue problem \eqref{55}, we have that $Ker(\cl)=Ker(\cl_+)\oplus Ker(\cl_-)=span[\nabla \phi_\om, \phi_\om]$. The generalized eigenvectors behind $\p_1\phi_\om, \p_2 \phi_\om$ are
  $$
  \psi_j=\left(\begin{array}{c}
  0 \\ \cl_-^{-1}(\p_j \phi_\om)
  \end{array} \right), j=1,2,
  $$
while $\psi_3=\left(\begin{array}{c}
  \cl_+^{-1}(\phi_\om) \\ 0
  \end{array} \right)$, if $\phi_\om \perp Ker(\cl_+)$. Note that such a property, which is usually referred to weak non-degeneracy of $\phi_\om$ while likely true, is not established here.

  Consequently, the matrix $\cd$ split into two symmetric sub-matrices -    $(\dpr{\cl_-^{-1}(\p_i \phi_\om)}{\p_j \phi_\om})_{i,j\in \{1,2\}}$ and $\dpr{\cl\psi_3}{\psi_3}=\dpr{\cl_+^{-1} \phi_\om}{\phi_\om}$, provided $\phi_\om$ is weakly non-degenerate.  Since $\cl_-\geq 0$, it is clear that $n(\cd)=
  n(\dpr{\cl_+^{-1} \phi_\om}{\phi_\om})$. In addition, as we have established $n(\cl)=n(\cl_-)+n(\cl_+)=1$. Thus, the formula \eqref{se:10} (note that it must be that $k_c=k_0^{\leq 0}=0$), together with the other considerations laid out in this section,  imply the following result.
  \begin{corollary}
  	\label{cor:28}
  	Assume that $\phi_\om$ is weakly non-degenerate and $n(\cl_+)=1$.  For the spectral problem \eqref{55}, spectral stability is equivalent to $\dpr{\cl_+^{-1} \phi_\om}{\phi_\om}<0$. Equivalently, all waves with \\ $\dpr{\cl_+^{-1} \phi_\om}{\phi_\om}>0$ are unstable. Moreover, in cases of instability, it  manifests itself through a single real positive eigenvalue in $\si(J L)$.
  \end{corollary}
  Interestingly, one may provide a necessary and sufficient condition, even when $\phi_\om$ is not necessarily weakly non-degenerate. Indeed, this is done in Theorem 4.1, \cite{Pel}. The result is that the spectral problem \eqref{55} is spectrally stable if and only if
  \begin{equation}
  \label{17}
  \cl_+|_{\{\phi_\om\}^\perp}\geq 0.
  \end{equation}
 This justifies the following sufficient condition for spectral instability.
 \begin{cor}
 	\label{cor:30} If $n(\cl_+)=1$ and there exists $\Psi\perp \phi_\om$, so that $\dpr{\cl_+\Psi}{\Psi}<0$, then the wave $\phi_\om$ is spectrally unstable, with exactly one unstable  real mode.
 \end{cor}

  \subsection{Proof of  the spectral instability of the waves}
  Per the results of Corollary \ref{cor:30}, it suffices to construct $\Psi\perp \phi_\om$, so that $\dpr{\cl_+\Psi}{\Psi}<0$. To this end, set
  $$
  \Psi:=x\cdot \nabla \phi_\om + \phi_\om.
  $$
  A direct calculation shows $\Psi\perp \phi_\om$. Indeed,
   $$
   \dpr{\Psi}{\phi_\om}=\dpr{x\cdot \nabla \phi_\om}{\phi_\om}+\|\phi_\om\|^2=
    \sum_{j=1}^2 \f{1}{2} \int_{\rtwo} x_j \p_j \phi_\om^2 +\|\phi_\om\|^2=0.
   $$
  It remains to calculate $\dpr{\cl_+\Psi}{\Psi}$. Since $-\De(x\cdot \nabla f)=-x\cdot \nabla \De f - 2 \De f$ and using the profile equation \eqref{20}, we compute
  $$
  \cl_+(x\cdot \nabla \phi_\om)=-2\De \phi_\om.
  $$
  So,
  \begin{eqnarray*}
  \dpr{\cl_+\Psi}{\Psi} &=& \dpr{\cl_+ (x\cdot \nabla \phi_\om)}{x\cdot \nabla \phi_\om+\phi_\om}+\dpr{\phi_\om}{\cl_+ (x\cdot \nabla \phi_\om)}+\dpr{\cl_+\phi}{\phi} = \\
  &=& -2 \dpr{\De \phi_\om}{x\cdot \nabla \phi_\om}+4 \|\nabla \phi_\om\|^2 - 2\int_{\rtwo} g'(\phi_\om^2) \phi_\om^4 dx=  4 \|\nabla \phi_\om\|^2 - 2\int_{\rtwo} g'(\phi_\om^2) \phi_\om^4 dx\\
  &=& \int_{\rtwo} \left(4 f(\phi_\om) \phi_\om - 8 F(\phi_\om)  -2  g'(\phi_\om^2) \phi_\om^4\right)dx =\\
  &=& -  \int_{\rtwo}  \left[e^{4\pi \phi^2}\left(8\pi \phi_\om^4+\f{1}{\pi}-4\phi_\om^2  \right)-\f{1}{\pi} \right]dx.
  \end{eqnarray*}
  where we have used that $\dpr{\De \phi_\om}{x\cdot \nabla \phi_\om}=0$, $\cl_+ \phi_\om=-2\phi_\om^3 g'(\phi^2)$ and the Pohozaev identities \eqref{42} and
  \eqref{45}. We will show momentarily  that the integrand function,
  \begin{equation}
  \label{187}
  e^{4\pi z^2}\left(8\pi z^4+\f{1}{\pi}-4 z^2  \right)-\f{1}{\pi}\geq 0
   \end{equation}
 whence
  $$
  \dpr{\cl_+\Psi}{\Psi} <0.
  $$
According to Corollary \ref{cor:30}, this  implies the spectral instability of $\phi_\om$.

 It remains to prove \eqref{187}. Indeed, it suffices to show that $\chi(x):=8\pi x^2 +\f{1}{\pi} -
4 x- \f{e^{-4\pi x}}{\pi}\geq 0$ for all $x\geq 0$. However, note that $f(0)=f'(0)=0$, while $f''(x)=16\pi(1-e^{-4\pi x})\geq 0$, whence $f(x)>0$ for all $x>0$.

\section{Instability by blow up}
Before tackling the strong instability of the standing waves, we need to make sure that the following Cauchy problem:
\begin{equation}
\label{3.1}
\left\{
\begin{array}{l}
i u_t+\triangle u+f_\mu (u)=0,\ (t,x)\in \mathbb{R}\times\mathbb{R}^{2} \\
\ \\
u(t,0)=u_0(x)
\end{array}
\right.
\end{equation}
where $u_0$ is an initial data in $H^{1}(\mathbb{R}^{2})$, has a unique solution for a time $T>0$.\\
$(\ref{3.1})$ has been resolved in \cite{COLL}. More precisely, the authors have proved the following result.

\begin{lem}[Theorem 1.10, \cite{COLL}]\ \\
	Let $u_0 \in H^{1}(\mathbb{R}^{2})$ such that $\Vert \nabla u_0\Vert_{0}<1$, then there exists a time $T>0$ and a unique solution to the Cauchy problem $(\ref{3.1})$ in the space $C_{T}(H^{1}(\mathbb{R}^{2}))$ with initial data $u_0$. Moreover $u\in L^{4}_{T}(C^{\frac{1}{2}}(\mathbb{R}^{2}))$, where $C^{\alpha}$ is the space of $\alpha$-H\"older continuous functions endowed with the norm $$\Vert u\Vert_{C^{\alpha}}=\Vert u\Vert_{\infty}+\sup_{x\neq y}\frac{\vert u(x)-u(y)\vert}{\vert x-y\vert^{\alpha}}.$$
	Additionally, if $T^*$ denote the maximal time, i.e,
	$
	T^{*}=\sup \left\lbrace T>0,\ (\ref{3.1})\ \text{has a solution on}\ [0,T]\right\rbrace,
	$
	we say that a solution blows up at$T^*$,  if  $\lim\limits_{t\nearrow T^{*}}\Vert \nabla u(t,.)\Vert^{2}_{2}=1.$
	We also have the mass and energy conservation  for the solutions of \eqref{3.1}, which take the form
	\begin{equation}
	\label{3.2}
	 M(u(t,.))=M(u_0), E_\mu(u(t,.))=E(u_0)
	 \end{equation}
\end{lem}
 For $v\in H^{1}(\mathbb{R}^{2})$, we define the action functional in the following way:
 $$
 S_{\mu}(v)=E_{\mu}(v)+\frac{1}{2}M(v)=\frac{1}{2}\Vert \nabla v\Vert_{2}^{2}+
 \frac{1}{2}\Vert v \Vert_{2}^{2}-\int F_{\mu}(v)dx.$$
 Note that a ground state as defined by $(1.3)$ minimizes the action functional in the following way:
  $$
  S_{\mu}(\Phi)=\inf\left\lbrace S_{\mu}(v):\ v\in H^{1}(\rtwo)\setminus\lbrace0\rbrace,\ P_{\mu}(v)=0\right\rbrace.
  $$
	The following two quantities also play a crucial role in the study of strong instability:
	$$
	P_{\mu}(v)=\frac{1}{2}\Vert v\Vert^{2}_{2}-\int F_{\mu}(v)dx.$$
	and
	$$
	I_{\mu}(v)=\frac{1}{2}\Vert \nabla v \Vert^{2}_{2}-\int v f_{\mu}(v)dx=2E_{\mu}(v)-\int v  f_{\mu}(v)-4F_{\mu}(v)dx.
	$$
Note that $S_{\mu}(\Phi)$ can also be rewritten in the following way :
$$
S_{\mu}(\Phi)=\inf\left\lbrace S_{\mu}(v):\ v \in H^{1}(\rtwo)\setminus\lbrace0\rbrace,\ I_{\mu}(v)=0\right\rbrace.
$$
	 A virial property has also been shown in \cite{DKM}, More precisely, the authors showed that if $u$ is a solution to $(\ref{3.1})$, then
	 $$\frac{d^{2}}{dt^{2}}\Vert xu\Vert^{2}_{2}=8I_{\mu}(u),\ \ \forall\ t\in[0,T^{*}).$$
	 Lastly, we introduce two classical sets that are important in the study of orbital stability as they have the desired invariance with respect to the flow of $(\ref{3.1})$:
	 $$
	 K^{-}_\mu=\left\lbrace v\in H^{1}(\rtwo)\backslash\lbrace 0\rbrace :\ S_{\mu}(v)<S_{\mu}(\Phi),\ I_{\mu}(v)<0\right\rbrace
	 $$
	 and
	 $$
	 K^{+}_\mu=\left\lbrace v \in H^{1}(\rtwo)\backslash\lbrace 0\rbrace :\ S_{\mu}(v)<S_{\mu}(\Phi),\ I_{\mu}(v)>0\right\rbrace
	 $$
	The following lemma provides sufficient conditions for a finite time blow-up.
	\begin{lem}[Lemma 3.9, \cite{DKM}]
		\label{le:4.2}
		Let $\mu\in\lbrace 0,1\rbrace$ and $u_{0}\in H^{1}(\mathbb{R}^{2})$ be such that $\Vert \nabla u_{0}\Vert_{2}^{2}<1$ and $E_{\mu}(u_{0})\geq 0.$ If
		$$
		u_{0}\in K^{-}_{\mu}, \ \ u_{0}\in H^{1}(\mathbb{R}^{2})\cap L^{2}(\vert x\vert^{2}dx),
		$$
		then the corresponding solution to $(\ref{3.1})$ blows up in finite time.
	\end{lem}
	 The next result, also proved in \cite{DKM}, shows that appropriate and close rescales of the ground state satisfy the requirements of Lemma \ref{le:4.2} and hence provide the instability by blow up claimed in Theorem \ref{theo:20}.
	 \begin{lem}[Lemma 3.13, \cite{DKM}]
	 	\label{le:4.3}
	 	Let $\mu\in\lbrace 0,1\rbrace$, $\Phi_{\lambda}(x):=\lambda \Phi (\lambda x)$. There exists $\epsilon>0$, so that  for all $\la:  1<\la<1+\epsilon$,  the following   holds true:
	 		\begin{eqnarray}
	 		\label{3.3}
	 		& &  E_{\mu}(\Phi_{\lambda})>0,\ I_{\mu}(\Phi_{\lambda})<0, \ \  S_{\mu}(\Phi_{\lambda})<S_{\mu}(\Phi),\\
	 		\label{3.4}
	 	& & 	\Vert \triangledown\Phi_{\lambda}\Vert^{2}_{2}<1, \ \  \Phi_{\lambda}\in H^{1}(\mathbb{R}^{2})\cap L^{2}(\vert x\vert^{2}dx)
	 		\end{eqnarray}
	 \end{lem}
	 {\bf Remark:} The precise statement of Lemma 3.13 in \cite{DKM} requires a slight modification, but  the result quoted here holds true due to the arguments presented there.
	
	 We are now ready to complete the proof of the strong instability of the waves $\Phi$. Consider $\Phi_n(x):=\la_n \Phi(\la_n x)$, for a sequence $\la_n\to 1+$. Clearly, $\lim_n \|\Phi_n-\Phi\|_{H^1}=0$. Also, the sequence $\Phi_n$ satisfies all the assumptions in Lemma \ref{le:4.2}, due to Lemma \ref{le:4.3}. Therefore, all solutions with initial data $\Phi_n$ blow up in finite time, whence we conclude the strong instability of the waves $\Phi$.



\begin{thebibliography}{99}
	
\bibitem{ALVES}C. O. Alves, M. A. Sauto, M. Montenegro, Existence of a ground state solution for a nonlinear scalar field equation with critical growth, {\emph Catc. Var. Partial. Differ. Equ.} {\bf 43}, (2012), p 537--554



\bibitem{Caz} T. Cazenave,  Semilinear Schr\"odinger equations. Courant Lecture Notes in Mathematics, {\bf 10},  American Mathematical Society, Providence, RI, 2003.

\bibitem{COLL} J. Colliander, S. Ibrahim, M. Majdoub, N. Masmoudi, Energy critical NLS in two space dimensions, {\em Journal of Hyper Diff. Eq.} {\bf 6},  No. 3,   (2009), p. 549--575.

\bibitem{DKM} V.D. Dinh, S. Keraani, M. Majdoub, \emph{Long time dynamics for the focusing nonlinear Schr\"odinger equation with exponential nonlinearities}, {\em Dyn. PDE}, {\bf 17}, (2020), No. 4, p. p. 551--589.

\bibitem{FO1} R. Fukuizumi,  M. Ohta, \emph{ Instability of standing waves for nonlinear Schr\"odinger equations with potentials}, {\em Differential Integral Equations}, {\bf 16},  (2003), no. 6, p. 691--706.

\bibitem{Hajaiej-Stuart} H. Hajaiej, C. A. Stuart, \emph{ Symmetrization inequalities for composition operators of Carathéodory type}, {\em Proc. London. Math. Soc}, {\bf 87}(2003),  p 396--418.


\bibitem{KKS} T. M. Kapitula, P. G. Kevrekidis, B. Sandstede, \emph{Counting eigenvalues
	via Krein signature in infinite-dimensional Hamitonial systems,}  {\em Physica D}, {\bf 3-4}, (2004), p. 263--282.

\bibitem{KKS2} T. Kapitula, P. G. Kevrekidis, B. Sandstede, \emph{ Addendum: "Counting eigenvalues via the Krein signature in
	infinite-dimensional Hamiltonian systems'' [Phys. D 195 (2004), no. 3-4,
	263--282] }
{\em Phys. D}  {\bf  201}  (2005),  no. 1-2, 199--201.


\bibitem{KP} T. Kapitula, K. Promislow, Spectral and Dynamical Stability of
Nonlinear Waves, 185, Applied Mathematical Sciences, 2013.

\bibitem{O1} M. Ohta, \emph{Strong instability of standing waves for nonlinear Schr\"odinger equations with harmonic potential},  {\em Funkcial. Ekvac.}, {\bf  61},
(2018), no. 1, p. 135--143.


\bibitem{Pel} D.E. Pelinovsky, Localization in Periodic Potentials: from Schr\"odinger
Operators to the Gross-Pitaevskii Equation, LMS Lecture Note Series 390 (Cambridge University Press, Cambridge, 2011)..

\bibitem{RS} 	B. Ruf,  F. Sani, \emph{Ground states for elliptic equations in $\rtwo$  with exponential critical growth},  Geometric properties for parabolic and elliptic PDE's, p. 251–267, Springer INdAM Ser., 2, Springer, Milan, 2013.


\bibitem{St1} A. Stefanov,  \emph{On the normalized ground states of second order PDE's with mixed power nonlinearities.}, {\em Comm. Math. Phys.}, {\bf 369},
(2019), no. 3, p. 929--971.


\bibitem{Wei} M. I. Weinstein {\emph Lyapunov stability of ground states of nonlinear dispersive evolutions equations}, {\em Comm. Pure Appl. Math.}, {\bf 39}, (1986), p. 51--67.

\end{thebibliography}
\end{document}